\documentclass{amsart}

\usepackage{graphics}
\usepackage{epsfig}
\usepackage[all]{xy}
\usepackage{graphicx}
\usepackage{amssymb}

\newtheorem{theorem}{Theorem}[section]
\newtheorem{lemma}[theorem]{Lemma}

\theoremstyle{definition}

\newtheorem{corollary}[theorem]{Corollary}
\newtheorem{proposition}[theorem]{Proposition}
\newtheorem*{theo*}{Theorem}
\newtheorem*{Conjecture*}{Conjecture}

\theoremstyle{remark}

\DeclareMathOperator{\Der}{Der}
\DeclareMathOperator{\Aut}{Aut}
\DeclareMathOperator{\ad}{ad}
\DeclareMathOperator{\End}{End}

\numberwithin{equation}{section}

\begin{document}

\title{Automorphisms and derivations of affine commutative and PI-algebras}

\author{Oksana Bezushchak}
\address{Faculty of Mechanics and Mathematics,
	Taras Shevchenko National University of Kyiv, 60, Volodymyrska street, 01033  Kyiv, Ukraine}
\email{bezushchak@knu.ua}
\thanks{The first author was supported by the PAUSE program (France), partly supported by UMR 5208 du CNRS; and partly supported by MES of Ukraine: Grant for the perspective development of the scientific direction "Mathematical sciences and natural sciences" at TSNUK}

\author{Anatoliy Petravchuk}
\address{Faculty of Mechanics and Mathematics,
	Taras Shevchenko National University of Kyiv, 60, Volodymyrska street, 01033  Kyiv, Ukraine}
\email{petravchuk@knu.ua, apetrav@gmail.com}
\thanks{The second author was supported by MES of Ukraine: Grant for the perspective development of the scientific direction "Mathematical sciences and natural sciences" at TSNUK}

%    Information for third author
\author{Efim Zelmanov}
\address{SICM, Southern University of Science and Technology, 518055 Shenzhen, China}
\email{efim.zelmanov@gmail.com}

%    General info

\date{\today}


\begin{abstract}
We prove analogs of A.~Selberg's result  for finitely generated subgroups of  $\text{Aut}(A)$ and of  Engel's theorem for subalgebras of  $\text{Der}(A)$ for a finitely generated associative commutative algebra $A$ over an associative commutative ring. We prove also an analog of the theorem of W.~Burnside and I.~Schur about locally finiteness of  torsion subgroups of $\text{Aut}(A)$.
\end{abstract}

\maketitle

\section{Introduction}

Let $\textbf{A}$ be the algebra of regular (polynomial)  functions on  an affine algebraic variety $V$ over an associative commutative ring $\Phi$ with  $1.$ 

The group of $\Phi$-linear automorphisms $\text{Aut}(\textbf{A})$ and the Lie algebra of $\Phi$-linear derivations $\text{Der} (\textbf{A})$ are referred to as the group of polynomial automorphisms of $V$ and the Lie algebra of vector fields on $V$, respectively.

When the variety $V$ is irreducible, i.e. the ring $\textbf{A}$ is a domain,  the group $\text{Aut}(K)$ of  automorphisms of the field $K$ of fractions of $\textbf{A}$ is called the group of birational automorphisms of $V$; and the Lie algebra $\text{Der} (K)$ of derivations of  $K$ is called the Lie algebra of rational vector fields on $V$.

Let $\mathbb{F}$ be the field. Then $\mathbb{F}[x_1,\ldots, x_n]$ and $\mathbb{F}(x_1,\ldots, x_n)$ are the polynomial algebra and the field of rational functions. The group $\text{Aut}(\mathbb{F}(x_1,\ldots, x_n))$ and the algebra $\text{Der}(\mathbb{F}(x_1,\ldots, x_n))$ (resp. $\text{Aut}(\mathbb{F}[x_1,\ldots, x_n])$ and $\text{Der}(\mathbb{F}[x_1,\ldots, x_n])$) are called the \emph{Cremona group} and the \emph{Cremona Lie algebra} (resp. \emph{polynomial Cremona group} and \emph{polynomial Cremona Lie algebra}).

Recall that a group is called \emph{linear} if it is embeddable into a group of invertible matrices over an associative commutative ring. Groups $\text{Aut}(\textbf{A})$ are, generally speaking, not linear. It has been an ongoing effort of many years to understand:
\begin{center}\emph{which properties of linear groups can be carried over to \\ automorphisms groups $\emph{\text{Aut}(\textbf{A})}$ and to Cremona groups}?\end{center}

J.-P.~Serre \cite{Serre38,Serre37} studied finite subgroups of Cremona groups. V.~L.~Popov \cite{Popov28} initiated the study of the question of whether the celebrated Jordan's theorem on finite subgroups of linear groups carries over to the groups  $\text{Aut}(A).$ For some important results in this direction see \cite{Bandman_Zarhin5,Birkar8,Deserti11,Popov28,Popov29,Prokhorov_Shramov31}.

S.~Cantat \cite{Cantat10} proved the Tits Alternative for Cremona groups of rank $2$.

In this paper, we prove analogs of A.~Selberg's result \cite{Selberg36} (see also \cite{Alperin2}) for finitely generated subgroups of  $\text{Aut}(A)$ and of  Engel's theorem for subalgebras of  $\text{Der}(A)$ for a finitely generated associative commutative algebra $A.$

We say that a group is \emph{virtually torsion free} if it has a subgroup of finite index that is torsion free.

\begin{theorem}\label{Theorem1} Let $A$ be a finitely generated associative commutative algebra over an associative commutative ring $\Phi$ with $1$. Suppose that $A$ does not have additive torsion. Then
	\begin{enumerate}
		\item[$(a)$] an arbitrary finitely generated subgroup of the group $\emph{\text{Aut}}(A)$ is virtually torsion free;
		\item[$(b)$] if $A$ is a finitely generated ring (i.e. $\Phi$ is the ring of integers $ \mathbb{Z}$), then the group $\emph{\text{Aut}}(A)$ is virtually torsion free.
\end{enumerate} \end{theorem}

\begin{corollary}[\rm\textbf{An analog of the theorem of W.~Burnside and I.~Schur}; see \cite{Jacobson17,Jacobson18}]\label{Cor1} Under the assumptions of theorem $\ref{Theorem1}(a)$ every torsion subgroup of  $\emph{\text{Aut}}(A)$ is locally finite.
\end{corollary}

\begin{corollary}\label{Corollary2} Every torsion subgroup of a polynomial Cremona group \\ $\text{Aut}(\mathbb{F}[x_1,\ldots,x_n]),$ where $\mathbb{F}$ is a field of characteristic zero, has an abelian normal subgroup of finite index.
\end{corollary}

Corollary \ref{Corollary2} immediately follows from corollary \ref{Cor1} and from the Jordan property of the group  $\text{Aut}(\mathbb{F}[x_1,\ldots,x_n]);$ see \cite{Birkar8,Prokhorov_Shramov31}.

If the torsion subgroup  in corollary \ref{Cor1} is torsion of bounded degree, then we don't need any assumptions on additive torsion. Indeed, in \cite{Bass_Lubotzky6}, it was shown that the group ${\rm Aut}(A)$ is locally residually finite. Hence, by the positive solution of the restricted Burnside problem (see \cite{Zelmanov41,Zelmanov42}), the group $G$ is locally finite.

Recall that a derivation $d$ of an algebra $A$ is called \emph{locally nilpotent} if for an arbitrary element $a\in A$ there exists an integer $n(a)\geq 1$ such that $d^{n(a)}(a)=0.$ For more information about locally nilpotent derivations see \cite{Freudenburg12}. An algebra is called \emph{locally nilpotent} if every finitely generated subalgebra is nilpotent.

Let $L\subseteq {\rm Der}(A)$ be a Lie algebra that consists of locally nilpotent derivations. The question of whether it implies that the Lie algebra $L$ is locally nilpotent was discussed in \cite{Freudenburg12,Petravchuk_Sysak26,Skutin39}. In particular, A.~Skutin \cite{Skutin39} proved local nilpotency of $L$ for a commutative domain $A$ of finite transcendence degree and characteristic zero.

\begin{theorem}\label{Theorem2}
	Let $A$ be a finitely generated  associative commutative algebra over an associative commutative ring, and let $L$ be a subalgebra of ${\rm Der}(A)$ that consists of locally nilpotent derivations. Then the Lie algebra $L$ is locally nilpotent.
\end{theorem}

The assumption of finite generation of the algebra $A$ is essential. If $A$ is the algebra of polynomials in countably many variables over a field, then there exists a non-locally nilpotent Lie subalgebra $L\subseteq {\rm Der}(A)$ that consists of locally nilpotent derivations. The following theorem, however, imposes a finiteness condition that is weaker than  finite generation.

Let $A$ be a commutative domain. Let $K$ be the field of fractions of $A.$ An arbitrary derivation of the domain $A$ extends to a derivation of the field $K,$ ${\rm Der} (A)\subseteq {\rm Der} (K).$ We have $K{\rm Der} (K)\subseteq {\rm Der} (K),$ hence ${\rm Der} (K)$ can be viewed as a vector space over the field $K.$

\begin{theorem}\label{Theorem3} Under the assumptions above, let $L\subseteq {\rm Der}(A)$ be a Lie ring that consists of locally nilpotent derivations. Suppose that ${\rm dim}_{K}KL<\infty .$ Then the Lie ring $L$ is locally nilpotent.
\end{theorem}

A special case of this theorem was proved by A.~P.~Petravchuk and K.~Ya.~Sysak in \cite{Petravchuk_Sysak26}.

The proof of theorem \ref{Theorem3} is based on  a stronger version of theorem \ref{Theorem2}, which is of independent interest.

Recall that a subalgebra $B$ of an associative commutative algebra $A$ is called an \emph{order} in $A$ if there exists a multiplicative semigroup $S\subset B$ such that
\begin{enumerate}
	\item[(1)] every element from $S$ is invertible in $A,$
	
	\item[(2)] an arbitrary element $a\in A$ can be represented as $a=s^{-1}b,$ where $s\in S$ and $b\in B.$
\end{enumerate}

Let $L \subseteq {\rm Der}(A)$ be a subalgebra. The subset
$A_L=\{ a\in A \ | \  \text{for an arbitrary} \ d\in L \  \text{there exists an integer} \ n(d)\geq 1 \ {\rm such \ that} \ d^{n(d)}(a)=0\}$ is a subalgebra of the algebra $A.$

\begin{proposition}\label{Proposition1}
	Let $A$ be a finitely generated commutative domain. Let $L$ be a subalgebra of  ${\rm Der}(A).$ If the subalgebra $A_L$ is an order in $A$, then the Lie algebra $L$ is locally nilpotent.
\end{proposition}

To achieve a natural generality and to expand to noncommutative cases we extended theorems \ref{Theorem1} and \ref{Theorem2} to algebras with polynomial identities, i.e. $\text{PI}$-algebras; see \cite{Aljad_Gia_Procesi_Regev1,Belov_Rowen7,Rowen35}.

A $\text{PI}$-algebra is called \emph{representable} if it is embeddable in a matrix algebra over an associative commutative algebra. In \cite{Small40}, L. W. Small constructed an example of a finitely generated $\text{PI}$-algebra that is not representable.

\begin{theorem}\label{Theorem4}
	Let $A$ be a finitely generated representable $\emph{\text{PI}}$-algebra over an associative commutative ring. Suppose that $A$ does not have additive torsion. Then
	\begin{enumerate}
		\item[(a)] an arbitrary finitely generated subgroup of the group ${\rm Aut}(A)$ is virtually torsion free;
		
		\item[(b)] if $A$ is a finitely generated ring, then the group ${\rm Aut}(A)$ is virtually torsion free.
	\end{enumerate}
\end{theorem}

\begin{theorem}\label{Theorem5}
	Let $A$ be a finitely generated $\emph{\text{PI}}$-algebra over an associative commutative ring. Suppose that $A$ does not have additive torsion. Then an arbitrary torsion subgroup of ${\rm Aut}(A)$ is locally finite.
\end{theorem}

We remark that theorem \ref{Theorem5} does not contain assumptions on representability.

C.~Procesi \cite{Procesi30} proved local finiteness of torsion subgroups of multiplicative groups of $\text{PI}$-algebras.

\begin{theorem}\label{Theorem6}
	Let $A$ be a finitely generated $\emph{\text{PI}}$-algebra over an associative commutative ring. Let $L\subseteq {\rm Der}(A)$ be a subalgebra that consists of locally nilpotent derivations. Then the Lie algebra $L$ is locally nilpotent.
\end{theorem}

\section{Preliminaries}

In this section, we review some facts that will be used in  proofs.

\subsection{}\label{1.1} Theorems \ref{Theorem1}, \ref{Theorem2}, \ref{Theorem5} and \ref{Theorem6} were formulated for finitely generated associative commutative algebras over an associative commutative ring $\Phi .$  We will show that it is sufficient to assume $\Phi =\mathbb Z,$ that is  to prove the theorems for finitely generated rings. In particular, theorems \ref{Theorem1}(b) and \ref{Theorem4}(b) imply theorems \ref{Theorem1}(a) and \ref{Theorem4}(a), respectively.
We will do it for theorem \ref{Theorem6}. The arguments for theorems \ref{Theorem1}, \ref{Theorem2} and \ref{Theorem5} are absolutely similar.

Let $\Phi$ be an associative commutative ring and let $A$ be an associative $\text{PI}$-algebra over $\Phi $ (see \textbf{\ref{1.2}}) generated by elements $a_1, \ldots a_m;$ and $ A\ni 1.$ Let $L\subseteq \Der _{\Phi}(A)$ be a Lie subalgebra generated by derivations $d_1, \ldots , d_n.$ Suppose that every derivation of the $\Phi$-algebra  $L$ is locally nilpotent. Let $\Phi \langle x_1, \ldots , x_m\rangle$ be the free associative $\Phi$-algebra in free generators $x_1, \ldots , x_m.$ Then there exist elements $f_{ij}(x_1, \ldots , x_m), 1\leq i\leq n, 1\leq j\leq m,$ such that $d_i(a_j)=f_{ij}(a_1, \ldots , a_m).$

Let $A_1$ be the subring of $A$ generated by  elements $1, a_1, \ldots ,a_m$ and by all coefficients of the elements $f_{ij}(x_1, \ldots ,x_m).$ It is straightforward that the subring $A_1$ is invariant under $d_1, \ldots , d_n.$  Assuming that theorem \ref{Theorem6} is true for $\Phi =\mathbb Z$, there exists an integer $r\geq 1$ such that $L^r(A_1)=(0).$ In particular, $L^r(a_i)=(0), 1\leq i\leq m.$ Since the elements $a_1, \ldots , a_m$ generate the $\Phi$-algebra $A$ we conclude that $L^r=(0).$

\vspace{15pt}

Let us review some basic definitions and facts about $\text{PI}$-algebras that can be found in the books \cite{Aljad_Gia_Procesi_Regev1,Belov_Rowen7,Rowen35}.

\subsection{}\label{1.2} An associative algebra over an associative commutative ring $\Phi \ni 1$ is said to be \emph{$\text{PI}$}- if there exists an element
$$ f(x_1, \ldots , x_n)=x_1\cdots x_n+\sum _{1\not =\sigma \in S_n}\alpha _{\sigma}x_{\sigma (1)}\cdots x_{\sigma (n)}$$
of the free associative algebra $\Phi \langle x_1, \ldots , x_n\rangle$ such that $f(a_1, \ldots , a_n)=0$ for arbitrary elements $a_1, \ldots , a_n\in A;$ hereafter $S_n$ is the  group of permutations of the set $\{1,\ldots, n \}$. In this case we say that the algebra $A$ satisfies the identity $f(x_1, \ldots , x_n)=0.$

If $A$ is a $\text{PI}$-algebra, then it satisfies an identity with all the coefficients $\alpha _{\sigma}, 1\not =\sigma\in S_n,$ lying in $\mathbb Z.$ In other words, every $\text{PI}$-algebra is $\text{PI}$ over $\mathbb Z,$ i.e. $\text{PI}$ as a ring.

\subsection{}\label{1.6} A ring $A$  is called \emph{prime} if the product of any two nonzero  ideals is different from zero. If $A$ is a prime $\text{PI}$-ring, then the center $$Z=\{ a\in A \ | \ ab=ba \  \rm{for \  an \ arbitrary \ element} \ b\in A\}\not =(0)$$
and the ring of fractions $(Z\setminus ( 0))^{-1}A$
is a finite-dimensional central simple algebra over the field of fractions of the domain $Z$; see \cite{Markov22,Rowen34}.

\subsection{}\label{1.7}  A ring $A$  is called \emph{semiprime} if it does not contain nonzero nilpotent ideals. Let $A$ be a finitely generated semiprime $\text{PI}$-ring. Let $Z$ be the center of $A$ and let $Z^*$ denote the set of elements from $Z$ that are not zero divisors. Then the ring of fractions ${(Z^{*})}^{-1}A$ is a finite direct sum of simple finite-dimensional (over their centers) algebras.

\subsection{}\label{1.8} An element $a\in L$ of a Lie algebra $L$ is called ad-\emph{nilpotent} if the operator $${\rm ad} (a) : L\to L, \quad {\rm ad} (a): x \mapsto  [a, x],$$ is nilpotent.

Suppose that a Lie algebra $L$ is generated by elements $a_1, \ldots , a_m.$ Commutators in $a_1, \ldots , a_m$ are defined via the following rules:
\begin{enumerate}
	\item[(i)] an arbitrary generator $a_i, 1\leq i\leq m,$ is a commutator in $a_1, \ldots , a_m;$
	
	\item[(ii)] if $\rho '$ and $\rho ''$ are commutators in $a_1, \ldots , a_m$, then $\rho =[\rho ', \rho '']$ is a commutator in $a_1, \ldots , a_m.$
\end{enumerate}

An element $a\in L$ is called a \emph{commutator} in $a_1, \ldots , a_m$ if it is a commutator because of (i) and (ii).

A Lie algebra $L$ over an associative commutative ring $\Phi \ni 1$ is called $\text{PI}$ (\emph{satisfies a polynomial identity}) if there exists a multilinear element of the free Lie algebra
$$f(x_0, x_1, \ldots , x_n)=({\rm ad} (x_1)\cdots {\rm ad} (x_n) +\sum _{1\not =\sigma \in S_n}\alpha _{\sigma}{\rm ad} (x_{\sigma (1)})\cdots {\rm ad} (x_{\sigma (n)}))x_0,  \alpha _{\sigma}\in \Phi,$$
such that $f(a_0, a_1, \ldots a_n)=0$ for arbitrary elements $a_0, a_1, \ldots a_n\in L.$

The following theorem was proved in \cite{Zelmanov42}.

\begin{theo*}[\cite{Zelmanov42}] Let $L$ be a Lie $\emph{\text{PI}}$-algebra over an associative commutative  ring  generated by elements $a_1, \ldots , a_m.$ Suppose that every commutator in $a_1, \ldots , a_m$ is $\emph{\text{ad}}$-nilpotent. Then the Lie algebra $L$ is nilpotent.\end{theo*}

\section{Groups of automorphisms}

\begin{lemma}\label{Lemma1}
	Let $A$ be a finitely generated commutative domain without additive torsion. Then the group ${\rm Aut} (A)$ is virtually torsion free.
\end{lemma}

\begin{proof}
	Let $I$ be a maximal ideal of the ring $A.$ The field $A/I$ is finitely generated, hence $A/I$ is a finite field, $A/I\simeq GF(p^l).$ Let $\mathcal P$ be the set of all ideals $P\vartriangleleft A$ such that $A/P\simeq GF(p^l).$ Let $P_0$ be the ideal of the ring $A$ generated by all elements $a^{p^l}-a, a\in A,$ and by the prime number $p.$ It is easy to see that the ring $A/P_0$ is finite, $P_0\subseteq \cap _{P\in \mathcal P}P.$
	This implies that the set $\mathcal P$ is finite.
	
	Automorphisms of the ring $A$ permute ideals from $\mathcal P.$ The ideal $I$ belongs to $\mathcal P.$ Hence, there exists a subgroup $ H_1 \leq {\rm Aut} (A), |{\rm Aut} (A): H_1|<\infty ,$  that leaves the ideal $I$ invariant. We have $|A: I^2|<\infty .$ Therefore, there exists a subgroup $H_2\leq H_1, |{\rm Aut} (A):H_2|<\infty ,$ such that
	$$(1-h)(A)\subseteq I^2$$
	for an arbitrary element $h\in H_2.$ Furthermore, if $a_1, \ldots a_k\in I,$ then
	$$(h-1)(a_1\cdots a_k)=(h(a_1)-a_1+a_1)\cdots (h(a_k)-a_k+a_k)-a_1\cdots a_k=\sum b_1\cdots b_k,$$
	where each $b_i=(h-1)(a_i)$ or $a_i$ and in each summand at least one element $b_i$ is equal to $(h-1)(a_i).$ This implies that
	$$(1-h)(I^k)\subseteq I^{k+1}.$$
	By the Krull intersection theorem (see \cite{Atiyah_Macdonald4}), we have $$\bigcap _{k\geq 1}I^k=(0).$$ If an element from $H_2$ has  finite order, then this order must be a power of the prime number~$p.$
	
	Consider the ring $$\widetilde{A}=\langle 1/p , A\rangle \subseteq A\otimes _{\mathbb Z}\mathbb Q,$$ where $\mathbb Q$ is the field of rational numbers. If $\widetilde J$ is a maximal ideal of the ring $\widetilde A,$ then
	$$\widetilde{A}/\widetilde{J}\simeq GF(q^t)  \quad \text{for prime } \quad  q, \quad  q\not =p, \quad   \text{and} \quad \bigcap _{k\geq 1}\widetilde{J}\,^k=(0).$$
	Let $J=\widetilde{J}\cap A.$ Arguing as above, we find a subgroup $H_3\leq {\rm Aut} (A)$ of a finite index such that $(1-h)(J^k)\subseteq J^{k+1}, k\geq 0,$ for an arbitrary element $h\in H_3.$ Hence, if an element from $H_3$ has  finite order, then this order must be a power of the prime number $q.$
	
	Now, $H_2\cap H_3$ is a torsion free subgroup of ${\rm Aut} (A).$ This completes the proof of the lemma. \end{proof}

\begin{lemma}\label{Lemma2}
	Let $A$ be a semiprime finitely generated  associative commutative ring without additive torsion. Then the group ${\rm Aut} (A)$ is virtually torsion free.
	
\end{lemma}
\begin{proof}
	Let $S\subset A$ be the set of all nonzero elements that are not zero divisors. Then the ring of fractions $S^{-1}A$ is a direct sum of fields, $S^{-1}A=\mathbb{F}_1\oplus \cdots \oplus \mathbb{F}_k.$ An arbitrary automorphism of the ring $A$ extends to an automorphism of $S^{-1}A.$
	Hence, there exists a subgroup $H\leq {\rm Aut} (A)$ of finite index such that every automorphism from $H$ leaves the summands $\mathbb{F}_1, \ldots , \mathbb{F}_k$ invariant. For each $i, 1\leq i\leq k,$ the factor-ring
	$$K=A/A\cap ( \mathbb{F}_1\oplus \cdots \oplus \mathbb{F}_{i-1} \oplus \mathbb{F}_{i+1}\oplus \cdots \oplus \mathbb{F}_k)$$
	is a domain without additive torsion. By lemma \ref{Lemma1}, there exists a subgroup $H_i< H$ of  finite index such that the image of $H_i$ in $ {\rm Aut} (K)$ is torsion free. This implies that the group $\cap _{i=1}^kH_i$ is torsion free. Indeed, if an element $h\in \cap _{i=1}^kH_i$ has  finite order, then $h$ acts identically modulo $K,$ and we get
	$$(1-h)(A)\subseteq \bigcap _{i=1}^k(\mathbb{F}_1\oplus \cdots \oplus \mathbb{F}_{i-1} \oplus \mathbb{F}_{i+1}\oplus \cdots \oplus \mathbb{F}_k)=(0).$$
	This completes the proof of the lemma.
\end{proof}

\begin{proof}[Proof of theorem \emph{\ref{Theorem4}(b)}.] Let $A$ be a finitely generated representable $\text{PI}$-ring that does not have additive torsion. A.~I.~Malcev \cite{Malcev21} showed that the ring $A$ is embeddable in a matrix algebra over a field of characteristic zero, $A\hookrightarrow M_n(\mathbb{F}), \ {\rm char}\, \mathbb{F}=0.$ Let $a_1, \ldots , a_m$ be generators  of the ring $A,$ and let $\mathbb Z\langle X\rangle$ be the free associative ring on  free generators $x_1, \ldots , x_m.$ If $R\subseteq \mathbb Z\langle x_1, \ldots , x_m\rangle$ is a set of defining relations of the ring $A$ in the generators $a_1, \ldots ,a_m,$ then $A\simeq \langle x_1, \ldots , x_m \ | \ R=(0)\rangle .$
	
	Let $n, m\geq 2.$ Consider $m$ generic $n\times n$ matrices $$X_k=(x_{ij}^{(k)})_{1\leq i, j\leq n},   1\leq k\leq m.$$
	These are $n\times n$ matrices over the polynomial ring $\mathbb Z[X],$ where $$ X=\{x_{ij}^{(k)}, 1\leq i, j\leq n, 1\leq k\leq m\}$$ is the set of variables. The ring $G(m,n)$ generated by generic matrices $ {X}_1, \ldots , {X}_m$  is a domain and it is $\text{PI}$; see \cite{Amitsur3}.
	
	For a relation $r\in R$ let
	$$r(X_1, \ldots , X_m)=\big(r_{ij}(X)\big)_{1\leq i, j\leq n},\quad r_{ij}(X)\in \mathbb Z[X].$$
	Consider the  associative commutative ring $U$ presented by generators $X$ and  relations
	$r_{ij}(X)=0, $ $r\in R, $ $ 1\leq i, j\leq n,$ i.e. $$ U=\mathbb Z[X]/I, \quad
	I={\rm id}_{\mathbb Z[X]}\big(r_{ij}(X), \ r\in R, \ 1\leq i, j\leq n\big).$$
	
	Since the ring $A$ is embeddable in $M_n(\mathbb{F})$ it follows that the homomorphism $$u:A\to M_{n}(U), \quad u(a_k)=X_k+I\in M_n(U),\quad 1\leq k\leq m,$$ is an embedding. Moreover, the ring $U$ has the following universal property: \\ if $C$ is an associative commutative ring and $\varphi :A\to M_n(C)$ is an embedding, then there exists a unique homomorphism $U\to C$ that makes the diagram
	$$ \xymatrix{ {A}\ar [r] ^{u} \ar[rd]_{\varphi} &  {M_n(U)} \ar[d]  \\
		& {M_n(C)} }$$
	commutative.  
	
	This implies that every automorphism of the ring $A$ gives rise to an automorphism of the ring $U.$ Let
	$$T(U)=\{ x\in U \ | \ \text{there exists an integer} \ k\geq 1 \ \text{such  that} \ kx=0\}$$
	be the torsion part of the ring $U.$ Let $J\big(U/T(U)\big)$ be the radical of the ring $U/T(U),$ $J\big(U/T(U)\big)=J/T(U),$ where
	$$(0)\subseteq T(U)\subseteq J\vartriangleleft U, \quad \overline{U}=U/J.$$
	The factor-ring $\overline{U}$ is semiprime and does not have additive torsion. An arbitrary automorphism of the ring $A$ gives rise to an automorphism of $\overline{U}.$
	
	Since the ring $A$ is embeddable in $M_n(\mathbb{F})$, $\rm{char}\, \mathbb{F}=0,$ it follows that $A$ is embeddable in $M_n(\overline{U})$ and the group ${\rm Aut} (A)$  is embeddable in ${\rm \Aut} (\overline{U}).$ By lemma \ref{Lemma2}, the group ${\rm Aut} (\overline{U})$  is virtually torsion free and so is ${\rm Aut} (A).$ This
	completes the proof of theorem \ref{Theorem4}(b). \end{proof}

Recall that  theorem \ref{Theorem4}(b) implies theorems \ref{Theorem1} and \ref{Theorem4}(a).

We will discuss the annoying representability assumption  in theorem \ref{Theorem4}. Let $A$ be a finitely generated $\text{PI}$-algebra over the field of rational numbers $\mathbb Q,$ and let $J$ be the Jacobson radical of the algebra $A.$  By \cite{Braun9}, the Jacobson radical of a finitely generated $\text{PI}$-ring is nilpotent. So, the radical  $J$ is nilpotent. The stabilizer of the descending chain
$A\supset J\supset J^2\supset \cdots \ $  in $ {\rm Aut} (A)$ is torsion free.
Indeed, let $\varphi \in \Aut (A)$ and $(1-\varphi )J^i\subseteq J^{i+1}, i\geq 0.$ We assume that $\varphi ^{n}=1.$ Then we have
$$\varphi ^{n}=(\varphi -1+1)^n=\sum _{i=2}^n\binom{n}{i}(\varphi -1)^i+n(\varphi -1)+1.$$
Hence, $$n(1-\varphi )=\sum _{i= 2}^n\binom{n}{i}(\varphi -1)^i.$$
Suppose that $a\in A$ and $(1-\varphi )a\not=0.$ Let $(1-\varphi )a\in J^k\setminus J^{k+1}.$ By the above, $n(1-\varphi )a\in (\varphi -1)J^k\subseteq J^{k+1},$ a contradiction.

If the group ${\rm Aut} (A/J^2)$ is virtually torsion free, then so is the group ${\rm Aut} (A).$ Indeed, let $H$ be a torsion free subgroup of finite index in ${\rm Aut} (A/J^2)$ and let $\widetilde H$ be the preimage of $H$ under the homomorphism  ${\rm Aut} (A)\to {\rm Aut} (A/J^2).$ If $h\in \widetilde{H}$ is a torsion element, then $h$ acts identically modulo $J^2,$ hence $h$ stabilizes the chain $A\supset J\supset J^2\supset \cdots $ and $h=1.$  We proved that the subgroup $\widetilde H$ of ${\rm Aut} (A)$ is torsion free.

In all known examples of nonrepresentable finitely generated $\text{PI}$-algebras the Jacobson radical is nilpotent of degree $\geq 3.$

\begin{Conjecture*}
A finitely generated $\text{PI}$-algebra with $J^2=0$ is representable.
\end{Conjecture*}

If this conjecture is true, then the representability assumption in theorem \ref{Theorem4} can be dropped.

The analog of Selberg's theorem holds for automorphism groups of some algebras that are far from being $\text{PI}.$

\begin{proposition}\label{Proposition2}
	Let $A=\mathbb Z\langle x_1, \ldots, x_m\rangle, m\geq 2,$ be the free associative  ring on free generators $ x_1, \ldots, x_m.$ The group of automorphisms ${\rm Aut} (A)$ is virtually torsion free.
\end{proposition}

\begin{proof}
	Let $p$ be a prime number. Let $I_p$ be the ideal of the algebra $A$ generated by $p$ and by all elements $a^p-a, a \in A.$ The ideal $I_p$ is invariant under all automorphisms, the factor-ring $A/I_p$ is finite  and constant terms of all elements in $I_p$ are divisible by $p.$ Hence, $$\bigcap _{i\geq 1}I_{p}\,^{i}=(0).$$ The subgroup
	$$H_1={\rm ker} \big({\rm Aut} (A)\to {\rm Aut} (A/I_p^{2})\big)$$ has finite index in ${\rm Aut} (A)$ and every element of finite order in $H_1$ has an order, which is a power of $p.$ Now,  choose a prime number $q,$ $ p\not =q.$ The subgroup $$H_2={\rm ker} \big({\rm Aut} (A)\to {\rm Aut} (A/I_q^{2})\big)$$ also has  finite index in ${\rm Aut} (A)$ and every element of  finite order in $H_2$ has an order which is a power of $q.$ The subgroup $H_1\cap H_2$ is torsion free and has  finite index in ${\rm Aut} (A).$ This completes the proof of the proposition.
\end{proof}

\begin{lemma}\label{Lemma3}
	Let $A$ be a \emph{$\text{PI}$}-algebra. Let $\phantom{i}_{A}{M}$ be a finitely generated left $A$-module. Then the algebra of $A$-module endomorphisms of the module $\phantom{i}_{A}{M}$ is \emph{$\text{PI}.$}
\end{lemma}
\begin{proof}
	Let $M=\sum _{i=1}^{n}Am_i.$ Consider the free $A$-module $V$ on free generators $x_1, \ldots ,x_n:$ $$ V=\sum _{i=1}^{n}Ax_i,$$ and the homomorphism
	$$f: V \to M, \quad x_i  \mapsto m_i,\quad 1\leq i\leq n.$$
	Denote its kernel as  $V_0.$ Let
	$$E_1=\{ \varphi \in  {\rm End} _{A}(V) \ | \ \varphi (V_0)\subseteq V_0\}, \quad E_2=\{ \varphi \in  {\rm End} _{A}(V) \ | \ \varphi (V)\subseteq V_0\}.$$
	Then
	$$ {\rm End} _{A}(M)\simeq E_1/E_2.$$
	The algebra ${\rm End}_A(V)$ is isomorphic to the algebra of $n\times n$ matrices over $A$. Hence, ${\rm End} _A(V)$ is a $\text{PI}$-algebra. This implies that $E_1$ and $E_1/E_2$ are $\text{PI}$-algebras.
\end{proof}

\begin{proof}[Proof of theorem \emph{\ref{Theorem5}}.]  Let $A$ be a finitely generated $\text{PI}$-algebra over $\mathbb Q$, and let $G$ be a  finitely generated torsion subgroup of ${\rm Aut} (A).$ Consider  the Jacobson radical $J$ of the algebra $A.$ The semisimple algebra $\overline{A}=A/J$ is representable; see \cite{Herstein16}. Hence, by theorem \ref{Theorem4}(a), the group ${\rm Aut} (\overline{ A})$ has Selberg's property, and the image of the group $G$ in ${\rm Aut} (\overline{ A})$ is finite. In other words, the subgroup
	$ H=\{\varphi \in G \ | \ (1-\varphi )(A)\subseteq J\}$ has  finite index in $G.$
	
	Consider the subgroup $$K=\{\varphi \in {\rm Aut} (A) \ | \ (1-\varphi )(A)\subseteq J^2\}.$$ We showed that this subgroup centralizes the descending chain $A\supset J\supset J^2\ldots ,$ hence $K$ is a torsion free group. Therefore, $G\cap K=(1)$, and the homomorphism $G\to {\rm Aut} (A/J^2)$ is an embedding. Without loss of generality, we will assume that $J^2=(0).$ The radical $J$ can be viewed as an $\overline A$-bimodule.
	
	Let $a_1, \ldots , a_m$ be generators of the algebra $A,$ and let $h_1, \ldots , h_r$ be generators of the subgroup $H.$ We have $(1-h_i)(A)\subseteq J, J^2=0,$ hence $1-h_i$ is a derivation of the algebra $A.$ This implies that $(1-h_i)(A)$ lies in the $\overline A$-subbimodule of $J$ generated by elements  ${(1-h_i)(a_1), \ldots , (1-h_i)(a_m).}$  Let $J'$ be the $\overline A$-subbimodule of $J$ generated by elements ${(1-h_i)(a_j), 1\leq i\leq r, 1\leq j\leq m.}$ The finitely generated subbimodule $J'$ is invariant with respect to the action of $H.$ For an automorphism $h\in H,$ consider  the restriction  ${\rm Res}(h)$ of $h$ to $J'.$  This restriction is a bimodule automorphism of the $\overline A$-bimodule $J'.$ The mapping $$\varphi : H\to GL(_{\overline A}J'_{\overline A}), \quad h\mapsto {\rm Res}(h),$$ is a homomorphism to the group of bimodule automorphisms $GL(_{\overline A}J'_{\overline A}).$
	The $\overline A$-bimodule $J'$ is a left module over the algebra ${\overline A}\bigotimes _{\mathbb Q}{\overline A}^{op}$ and
	$$ GL(_{\overline A}J'_{\overline A})=GL_{{\overline A}\bigotimes _{\mathbb Q}{\overline A}^{op}}(J').$$
	The algebra ${\overline A}\bigotimes _{\mathbb Q}{\overline A}^{op}$ is $\text{PI}$; see \cite{Regev33}. By lemma \ref{Lemma3}, the algebra $${\rm End} _{{\overline A}\bigotimes _{\mathbb Q}{\overline A}^{op}}(J')$$ is $\text{PI}$ as well. Thus,  $\varphi (H)$ is a finitely generated torsion subgroup  of the  multiplicative group of a $\text{PI}$-algebra. By the result of C.~Procesi \cite{Procesi30}, the group $\varphi (H)$  is finite. The kernel $H'={\rm ker}\, \varphi$ is a subgroup of finite index in $G$ and for an arbitrary element $h\in H'$ we have $(1-h)(A)\subseteq J', (1-h)(J')=(0).$ Let $h^k=1, k\geq 1.$ We have
	$$1-h^k=k(1-h) \ {\rm mod} \ (1-h)^2.$$
	This implies $k(1-h)(A)=0$ and, therefore, $h=1, H'=(1).$
	Hence, $|G|<\infty .$ This completes the proof of the theorem.\end{proof}

\section{Lie rings of locally nilpotent derivations}

\begin{proposition}\label{Proposition3}
	Let $A$ be a finitely generated $\emph{\text{PI}}$-ring. Then the Lie ring ${\rm Der} (A)$ is $\emph{\text{PI}}.$
\end{proposition}
\begin{proof}
	For an integer $n\geq 2$ consider the following elements of the free Lie ring
	$$P_n(x_0, x_1, \ldots x_n)=\sum_{\sigma \in S_n}(-1)^{|\sigma |}\ad (x_{\sigma (1)})\cdots \ad (x_{\sigma (n)})x_0.$$
	For an associative commutative ring $\Phi$ let $W_{\Phi}(n)$ denote the Lie $\Phi$-algebra of $\Phi$-linear derivations of the polynomial algebra $\Phi [x_1, \ldots ,x_n].$ In \cite{Razmyslov32}, Yu.P.Razmyslov proved that for a field $\mathbb{F}$ of characteristic zero the Lie algebra $W_{\mathbb{F}}(n)$ satisfies the identity $P_N=0,$ where $N=(n+1)^{2}.$ The Lie ring $W_{\mathbb Z}(n)$ is a subring of the $\mathbb Q$-algebra $W_{\mathbb Q}(n).$ Hence, $W_{\mathbb Z}(n)$ satisfies the identity $P_N=0.$ Let $A$ be a $\text{PI}$-ring generated  by elements $a_1, \ldots , a_m.$ Since $A$ is a finitely generated $\text{PI}$-ring, it follows that  $A$ is an epimorphic image  of the ring of generic matrices $G(m, n)$ for some integers $m, n\geq 2;$ see \cite{Belov_Rowen7,Kemer19}. Let
	$$ G(m, n)\to A, \quad  X_k=\big(x_{ij}^{(k)}\big)_{1\leq i, j\leq n}\mapsto a_k, \quad 1\leq k\leq m,$$
	be an epimorphism. Let $N=(n^2m+1)^2.$ We will show that the Lie ring ${\rm Der} (A)$ satisfies the identity $P_N=0.$
	Denote $$X=\{\,x_{ij}^{(k)} \ | \ 1\leq i, j\leq n, \quad 1\leq k\leq m\, \}.$$ Choose derivations $d_0, d_1, \ldots , d_N\in {\rm Der} (A).$ There exist elements $f_{st}(x_1, \ldots , x_m)$ of the free associative ring $\mathbb Z\langle x_1, \ldots , x_m\rangle, $ $ 0\leq s\leq N,$ $ 1\leq t\leq m,$ such that
	$$d_s(a_t)=f_{st}(a_1, \ldots , a_m).$$
	Let $$f_{st}(X_1, \ldots , X_m)=\big(g_{ij}^{st}(X)\big)_{1\leq i, j\leq n},$$ where $g_{ij}^{st}(X)\in \mathbb Z[X]$ are entries of the matrix $f_{st}(X_1, \ldots X_m).$
	Consider  derivations $ \widetilde{d}_{s}$ of the ring $\mathbb Z[X],$  $$\widetilde{d}_{s}(x_{ij}^{(t)})=g_{ij}^{st}(X), \quad  1\leq i, j\leq n, \quad 0\leq s\leq N, \quad 1\leq t\leq m.$$
	Let $L$ be the Lie subring generated by the derivations $ \widetilde{d}_{s}, 0\leq s\leq N$ in ${\rm Der} (\mathbb  Z[X]).$  The mapping $ \widetilde{d}_{s}\to d_{s}, 0\leq s\leq N,$ extends to a homomorphism $L\to {\rm Der} (A).$ This implies $P_N(d_0, d_1, \ldots d_N)=0$ and completes the proof of the proposition.
\end{proof}

Now, our aim is to prove theorem \ref{Theorem6}. In view of \textbf{\ref{1.1}}, we will assume that the  finitely generated $\text{PI}$-algebra $A$ of theorem \ref{Theorem6} is a finitely generated ring.

\vspace{7pt}

Let's prove theorem \ref{Theorem6} and proposition \ref{Proposition1} for the case of prime characteristics.

Let $A$ be a finitely generated $\text{PI}$-ring and let $L\subseteq {\rm Der} (A)$ be a Lie ring that consists of locally nilpotent derivations. Suppose further that there exists a prime number $p\geq 2$ such that $pA=(0).$
	
	Let $a_1, \ldots , a_m$ be generators of the ring $A.$ Let $d\in L.$ There exists a power $p^k$ of the prime number $p$ such that $$d^{p^k}(a_i)=0, \quad 1\leq i\leq m.$$
	
	The power $d^{p^k}$ is again a derivation of the ring $A.$ Hence $d^{p^k}=0.$ This implies that ${\rm ad} (d)^{p^k}=0$ in the Lie ring $L.$ By proposition \ref{Proposition3},  the Lie ring $L$ is $\text{PI},$ and by  results of \cite{Zelmanov43}  (see \textbf{\ref{1.7}}), the Lie ring $L$ is locally nilpotent. Moreover, every finitely generated subalgebra $L_1$ of $L$ acts on $A$ nilpotently, i.e. there exists an integer $s\geq 1$ such that $$\underbrace{L_1\cdots  L_1}_{s}A=(0).$$ This proves theorem \ref{Theorem6} in the case of a prime characteristic.

\vspace{7pt}
	
	Now, let $A$ be an associative commutative ring generated by elements $a_1, \ldots , a_m,$  let $p$ be a prime number such that  $pA=(0),$ and let $L\subseteq {\rm Der} (A)$ be a Lie subring of ${\rm Der} (A).$ Suppose that the subring $A_L$ is an order in $A.$ Then $a_i=b_{i}^{-1}c_i, 1\leq i\leq m,$ where $b_i, c_i\in A_L.$ For an arbitrary derivation $d\in L$ there exists a power $p^k$ such that $d^{p^k}(b_i)=d^{p^k}(c_i)=0, 1\leq i\leq m.$ Then $d^{p^k}(a_i)=0, 1\leq i\leq m, $ and, therefore, $d^{p^k}=0.$ Again, by \cite{Zelmanov43}, the ring $L$ is locally nilpotent. This proves proposition \ref{Proposition1} in the case of prime characteristic.

\vspace{7pt}

A Lie ring $L$ is called \textit{weakly Engel} if for arbitrary elements $a, b\in L$ there exists an integer $n(a, b)\geq 1$ such that
$${\rm ad} (a)^{n(a, b)}b=0.$$
B.~I.~Plotkin \cite{Plotkin27} proved that a weakly Engel Lie  ring has a locally nilpotent radical. In other words, if $L$ is a weakly Engel Lie ring, then $L$ contains the largest locally nilpotent ideal $I$ such that the factor-ring $L/I$ does not contain nonzero locally nilpotent ideals. We denote $I=\text{Loc}(L).$

\begin{lemma}\label{Lemma4}
	Let $A$  be a finitely generated ring and let a Lie ring $L\subseteq {\rm Der} (A)$ consist of locally nilpotent derivations. Then the Lie ring $L$ is weakly Engel.
\end{lemma}
\begin{proof}
	Let the ring $A$ be generated by elements $a_1, \ldots , a_m.$ Let $d_1, d_2\in L.$ There exists an integer $n\geq 1$ such that $d_1^n(a_i)=0, 1\leq i\leq m.$ Since the set $$\{\, d_2d_1^i(a_j), \quad 0\leq i\leq n-1, \quad 1\leq j\leq m\,\}$$ is finite there exists an integer $k\geq 1$ such that $$d_1^kd_2d_1^i(a_j)=0, \quad 0\leq i\leq n-1, \quad 1\leq j\leq m.$$
	We have $${\rm ad} (d_1)^sd_2=\sum_{i+j=s}(-1)^j\binom{s}{i}d_1^id_2d_1^j.$$ Hence
	$$(\ad (d_1)^{n+k-1}d_2)(a_j)=0, 1\leq j\leq m.$$
	This implies
	${\rm ad} (d_1)^{n+k-1}d_2=0$ and completes the proof of the lemma.
\end{proof}

\begin{lemma}\label{Lemma5}
	Let $A$ be a finitely generated associative commutative ring. Let $L \subseteq  {\rm Der} (A)$ be a Lie ring of derivations such that the subring $A_L$ is an order in $A.$ Then the Lie ring $L$ is weakly Engel.
\end{lemma}
\begin{proof}
	Let  $a_1, \ldots ,a_m$ be generators of the ring $A,$ let $a_i=b_i^{-1}c_i, 1\leq i\leq m, $ where $b_i, c_i \in A_L.$ Choose derivations $d_1, d_2 \in L.$ In the proof of  lemma \ref{Lemma4} we showed that there exists an integer $s\geq 1$ such that
	$$({\rm ad} (d_1)^sd_2)(b_i)=({\rm ad} (d_1)^{s}d_2)(c_i)=0, 1\leq i\leq m.$$
	Since $d'={\rm ad} (d_1)^sd_2$ is a derivation of the algebra $A$ it follows that $d'(a_i)=0, 1\leq i\leq m,$ and therefore $d'=0.$ This completes the proof of the lemma.
\end{proof}

\begin{lemma}\label{Lemma6}
	Let $A$ be a finitely generated semiprime $\text{PI}$-ring. Then there exists a family  of homomorphisms $A\to M_{n}(\mathbb Z/p\mathbb Z)$ into matrix rings over prime fields that approximates $A.$
\end{lemma}
\begin{proof}
	The ring $A$ is representable \cite{Herstein16}, i.e. it is embeddable into a ring of matrices over a finitely generated  associative commutative semiprime ring $C, \  A \hookrightarrow M_n(C).$  Hilbert's Nullstellensatz \cite{Atiyah_Macdonald4} implies that $C$ is a subdirect product  of finite fields. Hence, there exists a family of homomorphisms $\varphi _i : A\to M_n(\mathbb{F}_i),$ where $ \mathbb{F}_i$ are finite fields such that $\cap _{i}{\rm ker}\, \varphi _i=(0).$  If ${\rm char}\, \mathbb{F}_i=p,$ then  the field $\mathbb{F}_i$ is embeddable into a ring of matrices  over $\mathbb Z/p\mathbb Z.$ This completes the proof of the lemma.
\end{proof}

\begin{lemma}\label{Lemma7}
	Let $A$ be a finitely generated prime $\text{PI}$-ring. Let $Z$ be the center of $A$ and let $K$ be the field of fractions of the commutative domain $Z.$ Then ${\rm dim} _{K}K{\rm Der} (A)<\infty .$ \end{lemma}
\begin{proof}
	Let $a_1, \ldots , a_m$ be generators of the ring $A.$  As we have remarked in \textbf{\ref{1.6}} the ring of fractions $\widetilde A=(Z\setminus ( 0))^{-1}A$ is a finite-dimensional  central simple algebra over the field $K.$ Let ${\rm dim} _{K}\widetilde A=s.$ We will show that ${\rm dim}_KK{\rm Der} (A)\leq ms.$
	Choose $ms+1$ derivations $d_1, \ldots , d_{ms+1}$ of the ring $A.$  Consider the vector space $$V=\underbrace{\widetilde A\oplus \cdots \oplus \widetilde A}_{m}$$ over the field $K$, ${\rm dim} _KV=ms,$ and vectors $v_i=(d_i(a_1), \ldots , d_i(a_m))\in V, 1\leq i\leq ms+1.$ There exist coefficients $k_1, \ldots k_{ms+1}\in K,$ not all equal to $0,$ such that $$\sum _{i=1}^{ms+1}k_iv_i=0.$$  This implies $d(a_i)=0, 1\leq i\leq m,$ where $d=\sum_{i=1}^{ms+1}k_id_i.$ Since $d$ is a derivation of the ring $\widetilde A$ and elements $a_1, \ldots , a_m$ generate $A$ as a ring it follows that $d(A)=(0).$ This implies that $d(K)=0$ and completes the proof of the lemma.
\end{proof}

Now, we will prove theorem \ref{Theorem6} and proposition \ref{Proposition1} in the case when the algebra $A$ is prime.

As above, let $A$ be a finitely generated prime $\text{PI}$-ring, let $Z=Z(A)$ be the center of the ring $A,$ and let $ K=(Z\setminus \{0\})^{-1}Z$ be the field of fractions of the domain $Z.$ Suppose that a Lie ring $L\subseteq {\rm Der} (A)$ consists of locally nilpotent derivations. For a derivation $d\in L$ let ${\rm id}_L(d)$ denote the ideal of the Lie ring $L$ generated by the element $d.$ Consider the descending chain of ideals
$$ I_1=L, \ \
I_{i+1}=\sum _{d\in I_{i}} \ [{\rm id}_L(d), {\rm id}_L(d)].$$
Since ${\rm dim} _KKL<\infty$, by lemma \ref{Lemma7},  it follows that the descending chain
$$ KI_1\supseteq KI_2\supseteq \cdots $$ stabilizes. Let $KI_l=KI_{l+1}= \cdots.$ We will show that $I_l=(0).$ Indeed, there exists a finite collection of derivations $d_1, \ldots , d_r\in I_l$ such that
$$KI_{l+1}=\sum _{i=1}^rK[{\rm id}_L(d_i), {\rm id}_L(d_i)].$$
Recall that
\begin{equation}\label{equation1} {\rm id}_L(d_i)=\mathbb{Z}d_i+\sum _{t\geq 1} [\ldots [d_i, \underbrace{ L], L], \ldots , L}_{t}] .
\end{equation}
Let
\begin{equation}\label{equation2} d\in [{\rm id}_L(d_i), {\rm id}_L(d_i)]. 
\end{equation}
Expanding the commutators on the right-hand sides of (\ref{equation1}) and (\ref{equation2}) we get
$$d=\sum \underset{(\star _1)}{ . . . }d_i\underset{(\star _2)}{ . . . }d_i\underset{(\star _3)}{ . . . },$$
where $(\star _{1})$ is a product of derivations from $L$ and, possibly, a multiplication by an element from $K,$
$(\star _{2})$ and $(\star _{3})$ are products, may be empty, of derivations from $L.$ Hence,
$d=\sum \cdots d_i\cdots ,$ where each summand has a nonempty product of derivations from $L$ to the right of $d_i.$

Since $d_1, \ldots , d_r\in \sum _{j}K[{\rm id}_L(d_j), {\rm id}_L(d_j)],$ we have
\begin{equation}\label{equation3} d_i=\sum k_{ijt}u_{ijt}d_jv_{ijt}, \quad 1\leq i\leq r,
 \end{equation}
where $k_{ijt}\in K; u_{ijt}, v_{ijt}$ are products of derivations from $L;$ $v_{ijt}$ are nonempty products of derivations from $L.$

Let $b$ be a common denominator of all elements $k_{ijt},$ that is $k_{ijt}\in b^{-1}Z.$ Consider the finitely generated prime $\text{PI}$-ring $A_1=\langle b^{-1},  A\rangle.$ The ring $A_1$ is invariant under $\Der (A).$
Suppose that there exists an element $a\in A$ such that $d_i(a)\not =0.$ By lemma \ref{Lemma6},  there exists a family of homomorphisms $\varphi : A_1\to M_{n}(\mathbb Z/p\mathbb Z)$ that approximates the ring $A_1.$ Hence, there exists a prime number $p$ such that $d_i(a)\not \in pA_1.$

Consider the subring $L'$ of the Lie ring $L$ generated by all derivations that are involved in the products $v_{ijt}.$ Clearly, $L'$ is a finitely generated Lie ring. 

We have shown above that theorem \ref{Theorem6} is true for rings of prime characteristics. Applying this result to the ring $A/pA$, we conclude that the ring $L'$ acts nilpotently on $A/pA$. In other words,
there exists an integer $s\geq 1$ such that
\begin{equation}\label{equation4} \underbrace{L'\cdots L'}_{s}(A)\subseteq pA .
\end{equation}
Iterating (\ref{equation3}) $s$ times, we get
$$d_i=\sum u_td_jv_{i_1j_1t_1}\cdots v_{i_sj_st_s},$$
where $u_t\in A_1.$ By (\ref{equation4}), we get
$$ v_{i_1j_1t_1}\cdots v_{i_sj_st_s}A\subseteq pA\subseteq pA_1$$
and, therefore, $d_i(a)\in pA_1, 1\leq i\leq r,$
a contradiction. We showed that $I_l=(0).$ Recall that, by B.~I.~Plotkin's theorem \cite{Plotkin27}, the ring  $L$ has a locally nilpotent radical ${\rm Loc}(L).$ Let $i\geq 1$ be a minimal positive integer such that $I_i\subseteq {\rm Loc}(L), i\leq l.$ Suppose that $i\geq 2.$ For an arbitrary element $a\in I_{i-1}$ the ideal ${\rm id}_L(a)$ is abelian modulo $I_i.$ Since the factor-ring $L/{\rm Loc}(L)$ does not contain nonzero abelian ideals it follows that $a\in {\rm Loc}(L), I_{i-1}\subseteq {\rm Loc}(L),$ a contradiction.

We showed that $L=I_1\subseteq {\rm Loc}(L),$ in other words, the ring $L$ is locally nilpotent. This completes the proof of theorem \ref{Theorem6} in the case when the ring $A$ is prime.

\vspace{7pt}

To finish the proof of proposition \ref{Proposition1} we need just to repeat the arguments above. Let $A$ be a commutative domain, $L\subseteq {\rm Der} (A)$ and $A_L$ is an order in $A.$ We see that the subring $(A_1)_L$ is an order in the ring $A_1$ and, therefore, for any prime number $p$ the subring $(A_1/pA_1)_L$ is an order in $A_1/pA_1.$ In the case of a prime characteristic  proposition \ref{Proposition1} was proved for an arbitrary finitely generated  associative commutative ring, not necessarily a domain. Hence, we can apply it to $A_1/pA_1,$ and finish the proof of proposition \ref{Proposition1} following the proof of theorem \ref{Theorem6} verbatim.

To tackle the semiprime case we will need the following lemma.
\begin{lemma}\label{Lemma8}
	Let $A$ be a finitely generated semiprime ring. Then there exists a finite family of ideals $I_1, \ldots , I_n\vartriangleleft A$ such that each ideal $I_i, 1\leq i\leq n,$ is invariant under ${\rm Der} (A);$ each factor-ring $A/I_i$ is prime, and $$\bigcap _{j=1}^{n}I_j=(0).$$
\end{lemma}

\begin{proof} 	As we have mentioned in \textbf{\ref{1.7}}, the ring of fractions $\widetilde{A}=({Z^{\star}})^{-1}A,$ where $Z^{\star}$ is the set of all nonzero central elements  of $A$ that are not zero divisors, is a direct sum $\widetilde{ A}=\widetilde{ A_1}\oplus \cdots \oplus \widetilde{ A_n}$ of simple finite-dimensional over their centers algebras. Let
	$$ I_i=A\cap (\widetilde{ A}_1+\cdots +\widetilde{ A}_{i-1}+\widetilde{ A}_{i+1}+\cdots +\widetilde{ A}_{n}), 1\leq i\leq n.$$
	All direct summands $\widetilde{A}_i$ are invariant under ${\rm Der} (\widetilde{A}).$  An arbitrary derivation of the ring $A$ extends to a derivation of $\widetilde{A}.$ This implies that each ideal $I_i$ is invariant under ${\rm Der} (A).$
	
	Let us prove that each factor-ring $A/I_i$ is prime. Suppose that $a, b\in A$ and $aAb\subseteq I_i.$ We need to show that $a\in I_i$ or $b\in I_i.$  The inclusion above implies that
	$$a{\widetilde A}b\subseteq \widetilde{ A}_1+\cdots +\widetilde{ A}_{i-1}+\widetilde{ A}_{i+1}+\cdots +\widetilde{ A}_{n}.$$
	The factor-ring
	$${\widetilde A}/(\widetilde{ A}_1+\cdots +\widetilde{ A}_{i-1}+\widetilde{ A}_{i+1}+\cdots +\widetilde{ A}_{n})\simeq {\widetilde A}_{i}$$ is simple. Hence, at least one of the elements $a, b$ lies in $I_i.$ It is straightforward that $I_1\cap \cdots \cap I_n=(0).$ This completes the proof  of the lemma.\end{proof}

Now, we are ready to prove theorem \ref{Theorem6} in the case when the ring $A$ is semiprime. 

Let $A$ be a finitely generated semiprime $\text{PI}$-ring. Let $L$ be a finitely generated Lie subring  $L\subseteq {\rm Der}  (A)$ that consists of locally nilpotent derivations. Let $I_1, \ldots , I_n$ be the ideals of  lemma \ref{Lemma8}.  We showed above that there exists $r\geq 1$ such that
	$$L^r(A/I_i)=(0), 1\leq i\leq n.$$
	Hence, $$L^r(A)\subseteq \bigcap _{i=1}^nI_i=(0) \quad \text{and} \quad L^r=(0).$$ This completes the proof of theorem \ref{Theorem6} for semisimple rings.

\begin{lemma}\label{Lemma9}
	Let $A$ be a finitely generated $\text{PI}$-ring and let $L\subseteq {\rm Der} (A)$ be a Lie ring that consists of locally nilpotent derivations. Let $I\vartriangleleft A$ be a differentially invariant ideal  such that $I^2=(0)$ and the image of the Lie ring $L$ in ${\rm Der} (A/I)$ is locally nilpotent. Then the Lie ring $L$ is locally nilpotent.
\end{lemma}
\begin{proof}
	Choose derivations $d_1, \ldots, d_n\in L.$
	We need to show that the Lie ring $L'$ generated by $d_1, \ldots, d_n$ is nilpotent. By the assumption of lemma \ref{Lemma8}, there exists $r\geq 1$ such that $L'^r(A)\subseteq I.$ Let $d\in L'^r$ and let $a_1, \ldots , a_m$ be generators of the ring $A.$ There exists an integer $l\geq 1$ such that $d^l(a_j)=0, 1\leq j\leq m.$ Let $v=a_{i_1}\cdots \cdot a_{i_s}$ be a product of generators in the ring $A.$ Since $d(a_i)Ad(a_j)\subseteq I^2=(0)$ it follows that
	$$d^l(a_{i_1}\cdots  a_{i_s}) =d^l(a_{i_1})a_{i_2}\cdots a_{i_s}+a_{i_1}d^l(a_{i_2})\cdots a_{i_s}+\cdots +a_{i_1}\cdots a_{i_{s-1}}d^l(a_{i_s})=0.$$
	Hence $d^l=0.$
	
	Since the ring $L$ is weakly Engel by lemma \ref{Lemma4}, B.~I.~Plotkin's theorem \cite{Plotkin27} implies that the Lie ring $L'^r$ is    finitely generated. Hence, by \cite{Zelmanov43} (see also \textbf{\ref{1.7}}), the Lie ring $L'^r$ is nilpotent and the Lie ring $L'$ is solvable. Again by B.~I.~Plotkin's theorem, the Lie ring $L'$ is nilpotent. This completes the proof of the lemma.\end{proof}

Let us prove theorem \ref{Theorem6} in the case when the ring $A$ does not have additive torsion. 

Let $J$ be the Jacobson radical of the ring $A.$ By \cite{Braun9}, the radical $J$ is nilpotent. Let $J^n=(0), J^{n-1}\not =(0), n\geq 2.$ It is well known that if the ring $A$ does not have additive torsion, then the radical $J$ is differentially invariant.
	
	Let $$I=\{ a\in A \ | \ {\rm there \ exists \ an \ integer} \ s\geq 1 \ \  {\rm such \  that} \  sa\in J^{n-1}\}.$$
	The ideal $I$ is differentially invariant. We claim that $I^2=(0).$  Indeed, let $a, b\in I.$ There exist integers $s_1, s_2 \geq 1$ such that $s_1a\in J^{n-1}, s_2b \in J^{n-1}.$ Hence $s_1s_2ab \in (J^{n-1})^2=(0).$  Since the ring $A$ does not have additive torsion it follows that $ab=0.$
	
	The Jacobson radical of the ring $A/I$ is $J/I, (J/I)^{n-1}=(0).$  The ring $A/I$ obviously does not have additive torsion.  Hence, by inductive assumption on $n$, the image of $L$ in ${\rm Der} (A/I)$ is locally nilpotent; and by lemma \ref{Lemma9}, the ring $L$ is locally nilpotent.

\vspace{7pt}

Now, we are ready to finish the proof of theorem \ref{Theorem6}. 

Let $a_1, \ldots  , a_m$ be generators of a $\text{PI}$-ring $A.$ Let $L\subseteq {\rm Der} (A)$ be a finitely generated Lie subring such that every derivation from $L$ is locally nilpotent. Let $T(A)$ be the ideal of $A$ that consists of elements of a finite additive order. Clearly, $T(A)$ is differentially invariant. The factor-ring $A/T(A)$ does not have an additive torsion. Hence, by the proof of theorem~\ref{Theorem6} in the case when the ring $A$ does not have additive torsion, the image of the  ring $L$ in ${\rm Der} (A/T(A))$ is nilpotent. Therefore, there exists $r\geq 1$ such that for any derivation $d\in L^r$ we have $d(A)\subseteq T(A).$  Since the ring $L$ is finitely generated and weakly Engel by lemma \ref{Lemma4}, it follows from B.~I.~Plotkin's theorem \cite{Plotkin27}  that the Lie ring $L^r$ is finitely generated.
	
	We aim  to show that the Lie ring $L^r$ is nilpotent. Let $d_1', \ldots , d_l'$ be generators of $L^r.$ There exists an integer $n\geq 1$ such that $$nd_i'(a_j)=0, \quad 1\leq i\leq l, \quad 1\leq j\leq m.$$ Hence,
	$nL^r(A)=(0).$
	For a prime number $p$, consider the ideal
	$$I_p=\{ a\in A \ | \ {\rm there \ exists \ an \ integer } \ t\geq 1 \ {\rm such \ that } \ p^ta=0\}.$$ Let $a\in I_p,$ $ d\in L^r.$ Then $nd(a)=0$ and $p^td(a)=0$ for some $t\geq 1.$ Hence, for a prime number $p$ not dividing $n$, we have $L^rI_p=(0).$
	This allows us to consider the factor-ring $A/\sum_{p\nmid n}I_p$ instead of $A.$ In other words, we will assume that for a prime number $p$ not dividing $n$ the ring $A$ does not have a $p$-torsion.
	
	Let $p_1, \ldots , p_s$ be all distinct prime divisors of $n.$ Then $$T(A)=I_{p_1}\oplus \cdots \oplus I_{p_s}. $$ Let $s\geq 2.$ Inducting on the integer $n$ we can assume that the image of the Lie ring $L$ in each ${\rm Der} (A/I_{p_i})$ is nilpotent. In other words, there exists a number $r_i\geq 1$ such that $L^{r_i}(A)\subseteq I_{p_i}.$ This implies
	$$L^{\max{(r_1, r_2)}}(A)\subseteq I_{p_1}\cap I_{p_2}=(0).$$
	Therefore, we assume that $T(A)=I_p$ for some prime number $p.$ The ideal $pI_p$ lies in the Jacobson radical of $A$ and $pI_p$ is differentially invariant. Let $(pI_p)^q=(0), q\geq 1.$ If $q\geq 2,$ then inducting on $q$ we can assume that the image of the Lie ring $L$ in $\Der (A/(pI_p)^{q-1})$ is nilpotent. Hence, the ideal $(pI_p)^{q-1}$ satisfies the assumptions  of lemma \ref{Lemma9}. Suppose, therefore, that $q=1,$ $ pI_p=(0),$ $ n=p.$ Now, we have  $pL^r(A)=(0).$
	This implies that for an arbitrary derivation $d\in L^r$ every $p$-power $d^{p^k}$ is again a derivation. Indeed,
	$$ d^{p^k}(ab)=\sum _{i=0}^{p^k}\binom{p^k}{i}d^i(a)d^{p^k-i}(b)$$
	for  arbitrary elements $a, b\in A.$ If $0< i< p^k,$ then the binomial coefficient $\binom{p^k}{i}$ is divisible by $p,$  hence $$\binom{p^k}{i}d^i(a)=0, \quad \text{which implies} \quad d^{p^k}(ab)=d^{p^k}(a)b+ad^{p^k}(b).$$
	
	Choosing $d\in L^r$ and arguing as above, we find $p^k$ such that $d^{p^k}(a_j)=0, 1\leq j\leq m, $ therefore, $d^{p^k}=0.$ The Lie ring $L^r$ is finitely generated, $\text{PI}$, and an arbitrary derivation from $L^r$ is nilpotent. By \cite{Zelmanov43}, the Lie ring $L^r$ is nilpotent. The ring $L$ is solvable, hence, by the result of B.~I.~Plotkin \cite{Plotkin27}, it is nilpotent. This completes the proof of theorem \ref{Theorem6}.

\vspace{7pt}

Now, our aim is to prove theorem  \ref{Theorem3}. In the rest of this section, we assume that $A$ is a commutative domain; $L\subseteq  {\rm Der} (A)$ is a Lie ring that consists of locally nilpotent derivations; $K$ is the field of fractions of the domain $A,$ and ${\rm dim} _KKL<\infty.$
Our aim is to prove that the Lie ring $L$ is locally nilpotent. Let 
\begin{equation}\label{equation5} KL=\sum _{i=1}^{n}Kd_i, \quad  d_i\in L, \quad \text{and} \quad  [d_i, d_j]=\sum _{t=1}^nc_{ijt}d_t, \quad c_{ijt}=\frac{a_{ijt}}{b_{ijt}},   
\end{equation}
where $a_{ijt},$ $  b_{ijt}\in A.$ Enlarging the set $\{ d_1, \ldots , d_n\}$ if necessary we will assume that the derivations $d_1, \ldots , d_n$ generate $L, $  that is, $L={\rm Lie}_{\mathbb Z}\langle d_1, \ldots , d_n\rangle .$ Let $ d_{i_1} \cdots  d_{i_m}$  be a product in the associative ring of additive endomorphisms of the field $K.$  We call this product \textit{ordered} if $i_1\leq i_2\leq \cdots \leq i_m.$
Let $\mathcal P$ denote the set of all ordered products of  derivations $d_1, \ldots , d_n$ including the empty product, i.e. the identity operator.

\begin{lemma}\label{Lemma10}
	For an arbitrary element $a\in A$ the set of ordered products $v=d_{i_1}\cdots d_{i_m}\in \mathcal P$ such that $v(a)\not =0,$ is finite.
\end{lemma}

\begin{proof}
	Let $$v=d_1^{k_1}d_{2}^{k_2}\cdots d_n^{k_n}, \quad \text{where} \ k_i \ \text{are nonnegative integers.}$$ There exists an integer $q_n\geq 1$ such that $d_{n}^{q_n}(a)=0.$ Hence, if $v(a)\not =0,$ then $k_n<q_n.$ Similarly, there exists $q_{n-1}\geq 1$ such that $$d_{n-1}^{q_{n-1}}d_n^i(a)=0 \quad \text{for all}   \quad 0\leq i\leq q_{i-1}.$$ Hence, $v(a)\not =0$ implies $k_n<q_n, k_{n-1}<q_{n-1}$ and so on. This completes the proof of the lemma.
\end{proof}

Consider the set $C=\{ c_{ijt}\}_{i, j, t}\subset K$; see (\ref{equation5}).

\begin{lemma}\label{Lemma11}
	An arbitrary product $d_{i_1}\cdots d_{i_r}$ can be represented as
	$$d_{i_1}\cdots d_{i_r}=\sum \pm (v_1(c_1))\cdots (v_s(c_s))v_0,   $$
	where in each summand the operators $v_0,v_1, \ldots ,v_s$ lie in $\mathcal P$ and elements $c_1, \ldots , c_s$ lie in $C.$
	
\end{lemma}

\begin{proof}
For a product $v=d_{i_1}\cdots d_{i_r}$ let $l$ be the number of $1\leq k\leq r-1,$ such that $i_k>i_{k+1}.$ Let $\nu (v)=(r, l).$ We will compare pairs $(r, l)$ lexicographically and use induction on $\nu (v).$	Let $i=i_k>i_{k+1}=j.$ Then $$d_id_j=d_jd_i+\sum _{t}c_{ijt}d_t.$$ Clearly, $$\nu (d_{i_1}\cdots d_{i_{k-1}}d_jd_id_{i_{k+2}}\cdots d_{i_r})<\nu (v).$$ Consider the product $$d_{i_1}\cdots d_{i_{k-1}}c_{ijt}d_td_{i_{k+2}}\cdots d_{i_r}.$$ Commuting the element $c_{ijt}$ with derivations $d_{i_1}, \ldots , d_{i_{k-1}}$ we get $$d_{i_1}\cdots d_{i_{k-1}}c_{ijt}=\sum (v'(c_{ijt}))v'',$$ where $v', v''$ are products of derivations $d_{i_1}, \ldots , d_{i_{k-1}}$ of total length $k-1.$ Hence, $$d_{i_1}\cdots d_{i_{k-1}}c_{ijt}d_td_{i_{k+2}}\cdots d_{i_r}=\sum \pm (v'(c_{ijt}))v''d_td_{i_{k+2}}\cdots d_{i_r}.$$ In each summand the lengths of products $v'$ and $v''d_td_{i_{k+2}}\cdots d_{i_r}$ are less than $r.$ Applying the induction assumption to these products, we get the assertion of the lemma. \end{proof}

Consider the subring $\widetilde{A}$ of the field $K$ generated by the elements $$v(a_{ijt}), \quad  v(b_{ijt}),\quad v(b_{ijt})^{-1};\quad v\in \mathcal P; \quad i,\ j,\ t\geq 1.$$ By lemma \ref{Lemma10}, the ring $\widetilde A$ is finitely generated.

\begin{lemma}\label{Lemma12}
	The subring $\widetilde A$ is invariant under the action of $L.$
\end{lemma}

\begin{proof}
	For an arbitrarily ordered product of derivations $v\in \mathcal P$ we have
	$$v(b_{ijt}^{-1})=\sum \frac{1}{b_{ijt}^m}(v_1b_{ijt})\cdots (v_sb_{ijt}),$$
	where $m\geq 1; v_1, \ldots , v_s\in \mathcal P,$ and
	$$ v(c_{ijt})=v(a_{ijt}\cdot b_{ijt}^{-1})=\sum _{v', v''\in \mathcal P}v'(a_{ijt})v''(b_{ijt}^{-1}).$$
	These equalities imply $v(c_{ijt})\in \widetilde{A}.$ Now, by lemma \ref{Lemma11}, the ring $\widetilde A$  is invariant under the action of $L.$ \end{proof}

 The ring $\widetilde A$ is generated by elements $v(a_{ijt}), v(b_{ijt})\in A\cap \widetilde{A}$ and elements $v(b_{ijt})^{-1}.$  Hence, an arbitrary element of the ring $\widetilde A$ can be represented as a ratio $a/b,$ where $a, b\in A\cap \widetilde{A}.$ Hence, $A\cap \widetilde{A}$ is an order in the ring $\widetilde A,$ and the multiplicative semigroup $S$ being generated by elements $v(b_{ijt})\not =0.$
	
	By proposition \ref{Proposition1}, the image of the ring $L$ in $\End _{\mathbb Z}(\widetilde{A})$ is a nilpotent Lie ring. Hence, there exists an integer $r\geq 1$ such that $L^{r}(\widetilde{A})=(0).$  By lemma \ref{Lemma5} and Plotkin's theorem \cite{Plotkin27}, the Lie ring $L^r$ is finitely generated.
	Consider the subfield
	$$K_0=\big\{\, \alpha \in K \ |  \ L^r(\alpha )=(0)\, \big\}.$$
	The $K_0$-algebra $A'=K_0A\subseteq K$ is a domain.  The field $K_0$ is invariant under the action of $L,$   hence the $K_0$-algebra $A'$ is invariant as well.
	
	Let $L'$ be the image of the Lie ring $L^r$ in $\End _{\mathbb Z}(A').$ Since all the coefficients $c_{ijt}$ lie in $K_0$ the product $K_0L$ is a Lie ring and a finite-dimensional vector space over $K_0.$ This implies
	that $K_0L'$ is a finite-dimensional $K_0$-algebra. Now,  Petravchuk-Sysak theorem (see \cite{Petravchuk_Sysak26}) implies that $L'$ is a nilpotent Lie ring. Again, by lemma \ref{Lemma5} and B.~I.~Plotkin's theorem, the Lie ring $L$ is nilpotent. This completes the proof of theorem \ref{Theorem3}.

\vspace{7pt}

We will finish with examples showing that corollary \ref{Cor1} of theorem \ref{Theorem1} and theorem \ref{Theorem2} are wrong for countably generated algebras.  Let $\mathbb{F}$ be an arbitrary field and let $A=\mathbb{F}[x_1, x_2, \ldots ]$ be the polynomial algebra on countable many generators. We will construct

(i) a Lie algebra $L\subset {\rm Der} (A)$ that consists of locally nilpotent derivations and is not locally nilpotent,

(ii) a torsion group $G<{\rm Aut} (A)$ that is not locally finite.

Consider the countable-dimensional vector space $V=\sum _{i\geq 1}\mathbb{F}x_i.$ There exists a countable finitely generated Lie algebra $L$ such that  every operator ${\rm ad} (a), a\in L,$ is nilpotent, and the algebra $L$ has zero center (see \cite{Golod14,Lenagan_Smoctunowicz20}). The mapping $L\to \mathfrak{gl}(L),$ $ a\mapsto {\rm ad} (a),$ $ a\in L,$ is an embedding of the Lie algebra $L$ in $\mathfrak{gl}(L)$ and every linear transformation ${\rm ad} (a)$ from the image of $L$ is nilpotent. Therefore, we can suppose that $L\subseteq \mathfrak{gl}(V)$ and every linear transformation from $L$ is nilpotent.  An arbitrary linear transformation on $V$ is  a restriction of a derivation from $$\sum _{i\geq 1}V\frac{\partial}{\partial x_i}.$$ Hence, we assume that
$$L \subseteq  \sum _{i\geq 1}V\frac{\partial}{\partial x_i}\subset {\rm Der} (A).$$
Since every derivation from $L$ acts nilpotently on $V$ it follows that it acts locally nilpotently on $A.$  Similarly, there exists a finitely generated torsion group $G< {\rm Aut} (V)$ that is not locally finite (see \cite{Grigorchuk15,Novikov_Adyan23,Novikov_Adyan24,Novikov_Adyan25}). Every linear transformation $\varphi \in GL(V)$ uniquely extends to an automorphism $\widetilde{\varphi}\in \Aut (A).$ Thus the mapping $GL(V)\to {\rm Aut} (A),$ $\varphi \mapsto \widetilde{\varphi}, $ is an embedding of groups. Hence, $G$ is  a torsion not locally finite subgroup of ${\rm Aut} (A).$

\bibliographystyle{amsplain}

\end{document}